\documentclass[12pt,a4paper]{article}
\usepackage{amsmath}
\usepackage{amssymb}

\usepackage{mathtools}
\usepackage{amsthm}
\usepackage[usenames]{color}

\usepackage{color}

\usepackage{graphicx}

\DeclareFontFamily{U}{BOONDOX-calo}{\skewchar\font=45 }
\DeclareFontShape{U}{BOONDOX-calo}{m}{n}{
  <-> s*[1.05] BOONDOX-r-calo}{}
\DeclareFontShape{U}{BOONDOX-calo}{b}{n}{
  <-> s*[1.05] BOONDOX-b-calo}{}
\DeclareMathAlphabet{\mathcalb}{U}{BOONDOX-calo}{m}{n}
\SetMathAlphabet{\mathcalb}{bold}{U}{BOONDOX-calo}{b}{n}
\DeclareMathAlphabet{\mathbcalb}{U}{BOONDOX-calo}{b}{n}

\usepackage{urwchancal}
\DeclareFontFamily{OT1}{pzc}{}
\DeclareFontShape{OT1}{pzc}{m}{it}{<-> s * [1.10] pzcmi7t}{}
\DeclareMathAlphabet{\mathcalc}{OT1}{pzc}{m}{it}
\newcommand{\1}[1]{{\mathbf 1}_{\{#1\}}}

\newcommand{\N}{{\mathbb N}}

\newcommand{\tM}{{\widetilde{M}}}

\newcommand{\calF}{{\mathcal{F}}}
\newcommand{\calL}{{\mu}}
\newcommand{\calN}{{\mathcal{N}}}
\newcommand{\calZ}{{\mathcal{Z}}}

\allowdisplaybreaks

\newtheorem{theo}{Theorem}%[section]
\newtheorem{lem}[theo]{Lemma}

\newtheorem{cor}[theo]{Corollary}
\newtheorem{rem}[theo]{Remark}

\newcommand{\bbP}{\mathbb{P}}
\newcommand{\bbE}{\mathbb{E}}

\newcommand{\bbN}{\mathbb{N}}
\newcommand{\bbR}{\mathbb{R}}

\renewcommand{\ell}{\mathcalb{l}}
\newcommand{\err}{\mathcalc{r}}

\title{
On the large interaction asymptotics of the free energy density of the Ising chain with disordered centered external field}
\author{Orph\'ee Collin$^{1}$}

\begin{document}

\maketitle

{\footnotesize 
\noindent $^{~1}$Universit\'e Paris Cit\'e 
and Sorbonne Universit\'e, CNRS, Laboratoire de Probabilit\'es,
Statistique et Mod\'elisation, F--75013 Paris, France
\\
\noindent e-mail:
\texttt{collin@lpsm.paris}
}

\begin{abstract}
This article completes \cite{cf:GG22} by identifying explicitly the leading coefficient in the asymptotic development of the free energy density of the centered Random Field Ising Chain as the spin-spin interaction of the chain goes to infinity. The assumptions on the disorder law are much weaker than in \cite{cf:GG22}: only polynomial moments are required. 
\\[.3cm]\textbf{AMS subject classifications (2010 MSC):}
60K37,  %Processes in random environments
82B44, % Disordered systems
60K35, % Stat mech. type models
\\[.3cm]\textbf{Keywords:} random field Ising chain, disordered systems,  random matrix products
\\[.3cm] This work is licensed under a Creative Commons Attribution | 4.0 International licence (CC-BY 4.0, https://creativecommons.org/licenses/by/4.0/).
\end{abstract}

%Possible to add :
%\begin{itemize}
%\item counterexample : when h is not in L2, greedy trajectory
%\item a direct proof of $\vartheta^2/2J$ by controlling trajectories with respect to Fisher configuration.
%\item extension to finite range RFIC
%\item discussion on the non-centered case.
%\end{itemize}

%Questions 
%\begin{itemize}
%\item In the introduction, should $\mu$ have a first moment so that Kingman's theorem may be used?
%\item introduce $\eta_p$ and $\delta_p$ the exponents as functions of $p$?
%\item Talk about $J+F_\mu(J)$ and that it is even when $\mu$ is symmetric?
%\item Derive the free eneergy in $\vartheta$?
%\item For proof of theorem 3, another possibility is to control $\bbE[\log(1+e^{2(H_L-J})]$.
%\end{itemize}

%Remarks by Rafael on 9th april still to take into account 
%\begin{itemize}
%\item Fmu is nonnegative unclear to raphael?? Why ?
%\item in introduction, make more prominent the article Alea GG
%\item in introduction : Emphasize that assumptions are weaker, and that it is at least a little surprising that the coefficient is so explicit 
%\item  Also similar to block spin transform in renormalization group
%\end{itemize}

\section{Introduction}\label{sec:introduction}

%\subsection{Presentation of the model and result}

We denote by $\bbN:=\{1, 2, \dots, \}$ the set of all positive integers. Let $\mu$ be a law on $\bbR$, with a finite first moment. 
%, which we assume to be centered, with finite variance and different from $\delta_0$. 
We consider, under a probability measure denoted $\bbP_\mu$, an i.i.d.~sequence $h=(h_n)_{n\in\bbN}$ of real random variables following law $\mu$. 

\medskip

For every realization of the sequence $h=(h_n)_{n\in\bbN}$, for every $N\in\bbN$, every $(a, b)\in\{-1, +1\}^2$ and every $J\in\bbR^+$, we define the partition function of the Ising chain on $\{1, \dots, N\}$ with external field $h$, boundary conditions $(a,b)$ and spin-spin interaction $J$ as:
\begin{equation}\label{eq:defZN}
Z_{N, h}^{a,b}(J):=\sum_{\substack{\sigma\in \{-1, 1\}^{\{0, 1, \dots, N\}}\\ \text{s.t. } \sigma_0=a,\,   \sigma_N=b}} \exp\Big(-2J \sum_{n=1}^{N} \1{\sigma_n\neq \sigma_{n-1}}+\sum_{n=1}^{N} \sigma_n h_n\Big)\, .
\end{equation}
We say that \emph{configuration} $\sigma$ exhibits a \emph{spin flip} at position $n$ if $\sigma_{n}\neq \sigma_{n-1}$. 

Using Kingman's sub-additive ergodic theorem, it is straightforward to show the existence, for every $J\in\bbR$, of a real number $F_{\mu}(J)$ such that for almost every realization of $h$ (i.e. $\bbP_\mu$-almost surely), for every $(a,b)\in\{-1, +1\}^2$, we have
\begin{equation}\label{eq:convFdisc}
\frac1N \log Z_{N, h}^{a,b}(J) \underset{N\to\infty}{\longrightarrow} F_{\mu}(J)\, .
\end{equation}
The number $F_{\mu}(J)$ is called the (quenched) free energy density of the Random Field Ising Chain (RFIC) with disorder law $\mu$ and interaction $J$. We are interested in the asymptotics of $F_{\mu}(J)$ as $J\to \infty$.
%It can be expressed as:
%\begin{equation}
%F_{\mu}(J)=\lim_{N\to \infty} \frac1N \bbE\left[ \log Z_{N, h}^{a,b}(J)\right]\, .
%\end{equation}%USEFUL ?

From now on, we assume that $\mu$ has a finite second moment, is centered and is different from $\delta_0$. We denote its variance by $\vartheta^2:=\bbE_\mu [h_1^2]$, with $\vartheta \in (0, \infty)$.
%\begin{itemize}
%\item $\calL$ admits finite exponential moments: $\bbE[\exp(t h_1)]<\infty$ for all $t$ in a neighbourhood of 0;
%\item $\calL$ is centered: $\bbE[h_1]=0$;
%\item $\calL$ has a positive variance: we set $\vartheta \in(0, \infty)$ such that $\vartheta^2:=\bbE[h_1^2]$.
%\end{itemize}
%The first assumption yields in particular that $\calL$ has a finite variance, which we assume to be positive (i.e., we exclude the possibility that almost surely, for all $n\in\N$, $h_n=0$) and we set $\vartheta \in(0, \infty)$ such that
%\begin{equation*}
%\vartheta^2:=\bbE[h_1^2].
%\end{equation*}
We will consider in particular the case of the centered Gaussian law with variance $\vartheta^2$, which we denote by $\calN_{\vartheta^2}$.
The aim of this article is to show that
\begin{equation}\label{eq:mainaim}
F_{\calL}(J) \sim \frac{\vartheta^2}{2J}\, , \quad \text{ as } J\to \infty.
\end{equation}

\medskip

To prove \eqref{eq:mainaim} we will proceed in two steps. To begin with, we will cover the special case of disorder law $\calN_{\vartheta^2}$ through the use of a continuum analogue of the RFIC, introduced as a weak disorder limit in \cite{cf:MW68}, and for which the free energy density is known exactly. Then we will show \eqref{eq:mainaim} under general integrability assumptions on $\mu$ using a coarse-graining argument, together with the central limit theorem. We say that $\mu$ admits finite exponential moments, if $\bbE_\mu[\exp(t h_1)]<\infty$ for all $t$ in a neighbourhood of 0, and, for $p$ positive real number, that $\mu$ is in $L^p$ - or has a $p$-th moment - if $\bbE_\mu[|h_1|^p]<\infty$. We are going to show the following three theorems.

\begin{theo}[Gaussian disorder]\label{th:mainNv} 
As $J\to\infty$, 
\begin{equation} 
F_{\calN_{\vartheta^2}}(J)=\frac{\vartheta^2}{2J}+O\left(\frac{(\log J)^\frac12}{J^2}\right) \, .
\end{equation}
\end{theo}

\begin{theo}[Disorder with finite exponential moments]\label{th:mainL} If $\mu$ admits finite exponential moments, then
\begin{equation}
F_{\mu}(J)=\frac{\vartheta^2}{2J}+O\left(\frac{(\log J)^\frac13}{J^{\frac43}}\right)  \,, \quad \text{as } J\to \infty \, .
\end{equation}
\end{theo}

%We denote $\delta_p=\frac43$ for $p\in[3, \infty)$ and $\delta_p=2-\frac2p$ for $p\in[2, 3)$. We observe that $p\mapsto \delta_p$ is increasing and continuous on $[2, \infty)$, ranging over $[1, \frac43)$.
%We denote $\eta_p= \frac{4p}{3p+2}$ for $p\in [3, \infty)$ and $\eta_p=\frac{2p^2-2p}{p^2+p-1}$ for $p\in[2, 3)$. We observe that $p\mapsto \eta_p$ is increasing and continuous on $[2, \infty)$, ranging over $[\frac45, \frac43)$ and taking value 1 at $\frac{3+\sqrt{5}}2$.

\begin{theo}[Disorder with finite polynomial moments]\label{th:mainLpoly}
Let $p\in [3, \infty)$. If $\mu$ is in $L^p$, then
\begin{equation}\label{eq:mainLpoly>3}
\frac{\vartheta^2}{2J} +O\left(\frac{(\log J)^\frac13}{J^{\frac43}}\right) \le F_{\calL}(J) \le \frac{\vartheta^2}{2J}+O\left(\frac{1}{J^{\frac{4p}{3p+2}}}\right) \,, \quad \text{as } J\to \infty \, .
\end{equation}

Let $p\in [2, 3)$. If $\mu$ is in $L^p$, then
\begin{equation}\label{eq:mainLpoly<3}
\frac{\vartheta^2}{2J} +O\left(\frac{(\log J)^{\frac{p-1}{2p}}}{J^{\frac{2(p-1)}{p}}}\right) \le F_{\calL}(J) \le \frac{\vartheta^2}{2J}+O\left(\frac{1}{J^{\frac{2p(p-1)}{p^2+p-1}}}\right) \,, \quad  \text{as } J\to \infty \,.
\end{equation}

In particular, if $\mu$ is in $L^p$ for some $p>\frac{3+\sqrt{5}}2\approx 2.618$, then the error terms are $o\left(\frac1J\right)$ and \eqref{eq:mainaim} holds.
\end{theo}

%\begin{rem}
%Under the only hypothesis that $\calL$ has a finite second moment, we will also show that $F_{\calL}(J)=O\left(\frac1{J}\right)$.
%\end{rem}

Often, in statistical physics, properties concerning the typical configurations of the system can be extracted from the free energy density. We are now going to discuss some general properties of the free energy density of the RFIC, and use our results to derive bounds on the mean density of spin changes in the RFIC in the thermodynamic limit. 

Let us denote by $P^{a, b}_{N, h, J}$ the RFIC measure, i.e., the Gibbs measure on $\{-1, +1\}^N$ associated to the partition function $Z^{a, b}_{N, h}(J)$. A classical computation yields expressions for the expectation and the variance of the density of spin changes under the RFIC measure: %for every realization of $h$, for all $J>0$, $N\in\bbN$ and all $(a,b)\in\{-1, +1\}^2$, we have
\begin{equation}\label{eq:derisdenswalls}
E^{a, b}_{N, h, J}\left[\frac1N  \sum_{n=1}^{N} \1{\sigma_n\neq \sigma_{n-1}}\right] = - \frac12 \partial_J\left(\frac1N \log Z^{a, b}_{N, h}(J)\right) \, .
\end{equation}
and 
\begin{equation}\label{eq:deris2denswalls}
\mathrm{Var}^{a, b}_{N, h, J}\left[\frac1N  \sum_{n=1}^{N} \1{\sigma_n\neq \sigma_{n-1}}\right] = \frac1{4N} \partial_J^2\left(\frac1N \log Z^{a, b}_{N, h}(J)\right) \, .
\end{equation}
Since the left-hand sides are non-negative, we see that $J\mapsto \frac1N \log Z^{a, b}_{N, h}(J)$ is non-increasing and convex on $\bbR$. Taking the $\bbP_\mu$-almost sure limit, as $N\to \infty$, we obtain that also $F_\mu$ is non-increasing and convex on $\bbR$. We mention also that $F_\mu$ is non-negative: this can be seen by lower bounding the partition function by the contribution given in \eqref{eq:defZN} by $\sigma=(a, b,b,\dots, b, b)$; the resulting bound goes to 0 as $N\to\infty$ by the law of large numbers.% Also : expression of $F_\mu(0)$.

Furthermore, the two-by-two matrix $( Z_{N, h}^{a,b}(J))_ {(a, b)\in\{-1, +1\}^2}$ can be expressed (see \cite{cf:CGG19}) as a product of i.i.d. random matrices, distributed under a law determined by $\mu$. Consequently, Furstenberg's theory can be invoked: the free energy density $F_\mu(J)$ is the top Lyapunov exponent of this random matrix product and has an integral expression involving the invariant measure of a certain Markov chain on $(0, \infty)$. Then, using results by D. Ruelle (see \cite{cf:R79}, or also \cite{cf:D08}), $F_\mu$ is analytic on $(0, \infty)$. %[OR $\bbR$ ??].

As a consequence of convexity and twice differentiability, we can take the limit $N\to\infty$ in \eqref{eq:derisdenswalls} and \eqref{eq:deris2denswalls}. We see that the variance goes to 0 as $N\to\infty$, so we obtain, for all $J>0$ and all $(a,b)\in\{-1, +1\}^2$, $\bbP_\mu$-almost surely, the convergence
%get that the first derivative of $F_{\mu}$ is linked to the mean density of spin changes in the RFIC, in the thermodynamic limit: 
 
\begin{equation}\label{eq:derisdenswallslim}
\frac1N  \sum_{n=1}^{N} \1{\sigma_n\neq \sigma_{n-1}} \underset{N\to\infty}{\longrightarrow} -\frac12 F'_{\calL}(J)\,
\end{equation}
in $P^{a, b}_{N, h, J}$-probability.

From our main theorems, we derive the following asymptotic results on the limiting quantity appearing in \eqref{eq:derisdenswallslim}.
\begin{cor}\label{th:cormain} 
As $J\to\infty$, 
\begin{equation}\label{eq:cormain}
-\frac12 F'_{\mu}(J)=\frac{\vartheta^2}{4J^2}+O(\varepsilon_J) \, ,
\end{equation}
%where $\varepsilon_J$ is equal 
%to $\frac{(\log J)^\frac14}{J^{\frac52}}$ in the case of Gaussian disorder,
%to $\frac{(\log J)^\frac16}{J^{\frac{13}6}}$ in the case of exponential integrability,
%and, in the case of polynomial integrability, to $\frac1{J^{\frac{3+\eta_p}2}}$ for $p\in(\frac{3+\sqrt{5}}2, \infty)$ and to $\frac{1}{J^{\eta_p+1}}$ for $p\in[2,\frac{3+\sqrt{5}}{2}]$, where $\eta_p$ denotes the exponent in the upper bound in Theorem \ref{th:mainLpoly}.

where $\varepsilon_J$ is equal,
\begin{itemize}
 \item
 in the case of Gaussian disorder, to $\frac{(\log J)^\frac14}{J^{\frac52}}$,
 \item  
in the case of exponential integrability, to $\frac{(\log J)^\frac16}{J^{\frac{13}6}}$,
 \item  
and, in the case of polynomial integrability, denoting by $\eta_p$ the exponent in the upper bound in Theorem \ref{th:mainLpoly}, to $\frac1{J^{\frac{3+\eta_p}2}}$ for $p\in(\frac{3+\sqrt{5}}2, \infty)$ and to $\frac{1}{J^{\eta_p+1}}$ for $p\in[2,\frac{3+\sqrt{5}}{2}]$.
\end{itemize}

%(Gaussian disorder) As $J\to\infty$, 
%\begin{equation}\label{eq:cormaingaus}
%-\frac12 F'_{\calN_{\vartheta^2}}(J)=\frac{\vartheta^2}{4J^2}+O\left(\frac{(\log J)^\frac14}{J^{\frac52}}\right) \, .
%\end{equation}
%(Exponential integrability) As $J\to\infty$, 
%\begin{equation}\label{eq:cormaingen}
%-\frac12 F'_{\calL}(J)=\frac{\vartheta^2}{4J^2}+O\left(\frac{(\log J)^\frac16}{J^{\frac{13}6}}\right) \,.
%\end{equation}
%(Polynomial integrability) [ADAPT LOWER BOUND]
%As $J\to\infty$, 
%\begin{equation}\label{eq:cormaingenpoly}
%-\frac12 F'_{\calL}(J)=\frac{\vartheta^2}{4J^2}+O\left(\frac{1}{J^{\frac{3+\eta_p}2}}\right) \,.
%\end{equation} 
\end{cor}

\begin{proof}
By convexity, for every $J>0$ and every $a\in(0, J]$, we have
\begin{equation}\label{eq:cormainconvex}
\frac{F_{\calL}(J)-F_{\calL}(J-a)}{a} 
\le F'_{\calL}(J) 
\le \frac{F_{\calL}(J+a)-F_{\calL}(J)}{a}\,.
\end{equation}
In the Gaussian case, apply this with $a=J^\frac12 (\log J)^\frac14$ to derive \eqref{eq:cormain} from Theorem \ref{th:mainNv}. In the case of exponential integrability, apply it instead with $a=J^\frac56 (\log J)^\frac16$ to derive \eqref{eq:cormain} from Theorem \ref{th:mainL}. Finally, in the case of polynomial integrability, when $\eta_p >1$, take $a=J^{\frac{3-\eta_p}2}$ to derive \eqref{eq:cormain} from Theorem \ref{th:mainLpoly}, and when $\eta_p\le 1$, recall that $F_\mu'$ is non-positive, and use the lower bound in \eqref{eq:cormainconvex} with $a=J/2$ and the non-negativity of $F_\mu$.
\end{proof}

\bigskip

Let us now discuss the links with the existing literature.
In \cite{cf:NL86}, the free energy  density $F_\calL(J)$ is computed exactly for two specific choices of the law $\calL$. 
In \cite{cf:GG22}, under the assumptions that $\calL$ admits finite exponential moments, that $\calL$ has a density with respect to the Lebesgue measure and that this density is
% uniformly $\theta$-Hölder continuous with a $\theta\in(0,1]$,
uniformly Hölder continuous, a much more detailed asymptotic development for $F_{\calL}(J)$ is obtained, precisely:
\begin{equation} \label{eq:GG22}
F_{\calL}(J) = \frac{\kappa_1}{2J+\kappa_2} +O\left(\exp(-cJ)\right), \quad  \text{ as } J\to \infty,
\end{equation}
with constants $\kappa_1\in (0, \infty)$,  $\kappa_2\in\bbR$ and $c\in(0, \infty)$ depending on $\mu$.
A formula appears in \cite{cf:GG22} for the value of $\kappa_1$, but it is not very explicit: it is in terms of an integral involving the invariant measure of a certain Markov chain on $(0, \infty)$, see equation (3.12) in \cite{cf:GG22}. 
In this perspective, the contribution of the present work is to show that $\kappa_1$ is simply the variance of $\calL$, which comes as a surprise: one would rather only expect that it is close to the variance \emph{asymptotically} (in the weak disorder limit, i.e., when the variance goes to 0). We mention that in \cite{cf:NL86} the leading coefficient is not recognized as being the variance of $\calL$. Finally, we stress that our result holds in a much more general framework than that of \cite{cf:GG22}, regarding the assumptions required on $\calL$.

A description of the typical configurations of the RFIC was proposed in the physical literature, by D. Fisher, P. Le Doussal and C. Monthus, in \cite{cf:FLDM01}. The typical configurations are conjectured to be \emph{close} to a given configuration which we call \emph{the Fisher configuration}, obtained through a renormalization procedure. This configuration is determined by the disorder sequence $h$ and has a density of spin changes which is equivalent to $\frac{\vartheta^2}{4 J^2}$ as $J\to\infty$. For more details on this, see \cite{cf:CGH2}, where an estimation of the discrepancy density between typical configurations and the Fisher configuration is obtained. Our present results concerning the free energy density and the mean density of spin changes are in accordance with the predictions of Fisher, and are beyond the reach of \cite{cf:CGH2}.
In the present article, we do not make use of Fisher's approach. Let us however mention that it is possible to use the Fisher configuration to derive the lower bound $F_\mu(J)\ge \frac{\vartheta^2 +o(1)}{2J}$, as $J\to  \infty$: simply by lower bounding the partition function by the contribution given only by the Fisher configuration, whose spin domains are well-understood, at least to first order, see Theorem 3 in \cite{cf:CGH2}.

\section{Comparison of the discrete partition function with Gaussian disorder with the continuum partition function}

In this section, we are going to prove Theorem \ref{th:mainNv}. In order to do so, we turn to the continuous analogue of the model. The continuum RFIC was introduced in \cite{cf:MW68} at the level of the free energy as a weak disorder limit of the discrete RFIC. The partition function and the continuum model themselves have been introduced in \cite{cf:CGH1}. We now present the definitions of the partition function and of the (quenched) free energy density.

Consider, under a probability $\bbP$, a standard Brownian motion $B=(B_t)_{t\in[0, \infty)}$. 
For every $J\in\bbR$ and $a\in\{-1, +1\}$, denote by $\mathbf{P}_J^a$ the law of a Poisson process $(s_t)_{t\ge 0}$ taking values in $\{-1,1\}$, with $s_0=a$ and jump rate equal to $\exp(-2 J)$, and denote by $\mathbf{E}_J^a$ the associated expectation. Then, for every realization of $B$, every $\ell\ge 0$, every $(a, b)\in\{-1, +1\}^2$ and every $J\ge 0$, we define the partition function of the continuum Ising chain on $[0, \ell]$ with external field $B$, boundary conditions $(a,b)$ and interaction $J$ as:
\begin{equation}\label{eq:defZconti}
\calZ_{\ell, B}^{a,b}(J):=\exp(\ell e^{-2J}) \;  \mathbf{E}_J^a\left[ 
\exp \left( \int_0^\ell s_t \mathrm{d}B_t \right)\1{s_\ell=b}
\right]\, .
\end{equation}
For almost every realization of $B$, i.e. $\bbP$-almost surely, for every $(a, b)\in\{-1, +1\}^2$, we have (see \cite{cf:CGG19} together with Lemma 5.1 of \cite{cf:CGH1}):
\begin{equation}\label{eq:convFconti}
\frac1\ell \log \calZ_{\ell,J,B}^{a,b}(J) \underset{\ell \to \infty}{\longrightarrow} \calF(J)\, ,
\end{equation}
where $\calF(J)$ is called the (quenched) free energy density of the continuum RFIC with interaction $J$, and:
\begin{equation}\label{eq:devFconti}
\calF(J):= \frac{e^{-2J}K_{-1}(e^{-2J})}{K_0(e^{-2J})} \sim \frac{1}{2J+\log 2+\gamma_{EM}}+O(e^{-4J}), \quad \text{ as } J\to \infty.
\end{equation}
In this expression, $K_{-1}$ and $K_0$ are the modified Bessel functions of second kind of index respectively -1 and 0, and $\gamma_{EM}$ is the Euler-Mascheroni constant. Only the developement $\calF(J)=\frac{1}{2J}+O(\frac1{J^2})$ is relevant for our purpose.

For the sake of simplicity, we choose to prove for now Theorem \ref{th:mainNv} when $\vartheta=1$. Remark \ref{rem:theta} will be devoted to generalizing the proof to any $\vartheta\in(0, \infty)$. 

We are going to compare the partition function of the discrete Ising chain with disorder law $\calN_{1}$ with the continuum partition function we just introduced.
This comparison is based on the following heuristic. Splitting $\bbR^+$ into the segments $[n-1, n], n\in\bbN$, we observe that:
\begin{itemize}
\item inside each segment, the Poisson parameter of the number of jumps is $e^{-2J}$,
\item the disorder energies over the segments, i.e., the variables $B_{n}-B_{n-1}$, are i.i.d. and distributed under $\calN_1$.
\end{itemize}
Thus, it seems reasonable to modify the configurations by artificially moving the positions of the spin changes to their integer part (this means slightly modifiying the energy $\int_{0}^\ell s_t \mathrm{d} B_t$ collected): then, we see the \emph{discrete} partition function associated to the sequence $(B_n-B_{n-1})_{n\in\bbN}$ appear. Keeping in mind that we expect the mean number of jumps to be of order $\frac 1{J^2}$, it seems likely that the error concerning the collected disorder energy is not too big, and that this method should yield the desired estimate. 

%This approach is backed by the following considerations. In \cite{cf:CGH1} the typical configurations of the continuum RFIC measure (the Gibbs meaure associated to the continuum partition function introduced in \eqref{eq:defZconti}) have been studied, based on the description provided by \cite{cf:FLDM01}, and the results suggest that their spin domains have length of order $J^2$. Therefore, for any fixed segment of length 1, we heuristically assume the probability, under the continuum RFIC measure, of having a spin change in the segment to be small - roughly, dominated by $O(J^{-2})$ - hence the error term in the energy should remain of the same order.

First, we need to slightly generalize the definition of this continuum partition function. For every real numbers $0\le \ell \le \err$, we set: 
\begin{equation}\label{eq-defBgen}
\begin{aligned}
\calZ_{\ell, \err, B}^{a,b}(J)
& : =\exp((\err-\ell) e^{-2J}) \mathbf{E}_J^a\left[ 
\exp \left( \int_\ell^\err s_t  \mathrm{d} B_t \right)\1{s_\err=b} \Big| s_\ell=a
\right]\, .
\end{aligned}
\end{equation}
Observe that $\calZ_{\ell, B}^{a,b}(J)=\calZ_{0, \ell, B}^{a,b}(J)$ and that $\calZ_{\ell, \err, B}^{a,b}(J)=\calZ_{\err-\ell, \Theta^\ell B}^{a,b}(J)$ if we set  $\Theta^\ell B:=(B_{\ell+t}-B_{\ell})_{t\in[0, \infty)}$.

Following our heuristic, we rewrite the continuum partition function defined in \eqref{eq:defZconti} in the following way. Take $N$ a positive integer and to any (cadlag) trajectory $(s_t)_{0\le t \le N} \in \{-1, 1\}^{[0,N]}$, associate the configuration  $\sigma\in \{-1, 1\}^{\{0, 1, \dots, N\}}$ defined by:
 \[\sigma_k = s_{k}, \qquad k=0, \dots, N.\]
%Partitioning the set of (cadlag) trajectories $(s_t)_{0\le t \le N}$ according to this application, we rewrite the expectation in \eqref{eq:defZconti} to obtain:
To each $\sigma$ corresponds a set of trajectories $(s_t)_{0\le t \le N}$, we rewrite the expectation in \eqref{eq:defZconti} according to this partitioning to obtain:
\begin{equation}\label{eq-rewriteAregroupconti}
 \calZ_{N,B}^{a,b}(J)=
 \sum_{\substack{\sigma\in \{-1, 1\}^{\{0, 1, \dots, N\}}\\ \text{s.t. } \sigma_0=a,\,   \sigma_N=b}} 
 \prod_{n=1}^N  \calZ_{n-1, n, B}^{\sigma_{n-1}, \sigma_n}(J) \, .
\end{equation}
We define the sequence $h=(h_n)_{n\in\bbN}$ by setting $h_n:=B_{n}-B_{n-1}$. As already stressed, $h$ is i.i.d with law $\calN_{1}$. Our aim is to compare the above sum with:
\begin{equation}
Z_{N, h}^{a, b}(J)
=\sum_{\sigma\in \{-1, 1\}^{\{0, 1, \dots, N\}} : \sigma_0=a, \sigma_N=b}
\prod_{n=1}^N \exp(-2J \1{\sigma_{n-1} \neq  \sigma_n}+\sigma_n h_n)\, .
\end{equation}
so naturally we are going to compare the continuum partition function on the \emph{block} $[n-1, n]$, i.e $\calZ_{n-1, n, B}^{\sigma_{n-1}, \sigma_n}(J)$, with $\exp(-2J \1{\sigma_{n-1} \neq  \sigma_n}+\sigma_n h_n)$. This comparison will be provided by the next lemma, that without loss of generality we formulate only for the first block $[0, 1]$. 
We introduce the random variable
\begin{equation}\label{eq:defHconti}
H:=\max_{0\le s\le t\le 1} | B_s-B_t|\, .
\end{equation}

\begin{lem}\label{th:lemcompZconti}
 For every realization of $B$, for every $(a, b)\in\{-1, +1\}^2$, for every $J\ge 0$ and $M\ge 0$, we have
\begin{equation}\label{eq:lowbdZconti}
 \calZ_{1, B}^{a, b}(J) \ge \exp\left(-2(J+M) \1{a \neq  b}+bB_{1}-2(H-M)_+\right)\, ,
\end{equation}
and %if $J-M\ge 0$,
\begin{equation}\label{eq:uppbdZconti}
\calZ_{1, B}^{a, b}(J) \le \exp\left(-2(J-M) \1{a \neq  b}+b B_{1} +  2(H-M)_+ +e^{2(H-J)}\right)\,.
\end{equation}
\end{lem}
%Observe that as $J\to\infty$ we have $\varepsilon(J)=O(e^{-2J})$. 

\begin{proof} The proof is based on the following observation: for any (cadlag) trajectory $(s_t)_{0\le t \le 1} \in \{-1, 1\}^{[0,1]}$ satisfying $s_0=a$ and $s_{1}=b$, denoting by $j$ its number of jumps on $[0, 1]$, we have
\begin{equation}\label{eq:boundcontiFobs}
\left|\int_0^{1} s_t  \mathrm{d} B_t - b B_{1} \right|\le 2 j H\, .
\end{equation}
We first handle the lower bound. If $a=b$ (respectively, if $a\neq b$), we lower bound the partition function $\calZ_{1, B}^{a, b}(J)$ by restricting the expectation $\mathbf{E}_J^a$ to the set of trajectories with no jump in $[0, 1]$ (respectively  with exactly one jump in $[0, 1]$).
Recalling the factor in front of the expectation in \eqref{eq:defZconti}, we obtain: 
 \begin{equation}
  \calZ_{1, B}^{a, b}(J)
  \ge \exp(- 2(J+H)\1{a\neq b}+b B_{1}) \, .
 \end{equation}
Using $H\leq M+(H-M)_+$ we derive the desired lower bound.

We now turn to the proof of the upper bound. Using the observation \eqref{eq:boundcontiFobs}, we have:
\begin{equation}\label{eq:uppbdcontiforinterm}
   \calZ_{1, B}^{a, b}(J)
   \leq \sum_{j} \frac{(e^{-2J})^j}{j!}\exp({bB_1+ 2jH}) \, ,
\end{equation}
where the sum is over non-negative even integers if $a=b$ and over positive odd integers if $a\neq b$. In the first case, let us bound the sum by the sum over all non-negative integers, in the second case let us bound by the sum over all positive integers, and use $j! \ge (j-1)!$. We obtain :
\begin{equation}\label{eq:uppbdcontiinterm}
  \calZ_{1, B}^{a, b}(J)
   \leq \exp \left(b B_{1}-2(J-H)\1{a\neq b} + e^{2(H-J)}\right) \, .
 \end{equation}
Using $H\leq M+(H-M)_+$ we see that the upper bound in the lemma follows.
%where we used that if $a\ne b$ (respectively if $a=b$), then $j$ cannot be equal to 0 (respectively to 1). We note that $\1{a=b}+\1{a\neq b}  e^{-2J} e^{2M}= e^{-2(J-M)\1{a\neq b}}$ and, using $j!\ge (j-2)!$, that
%\begin{equation}
%\sum_{j\ge 2} \frac{(e^{-2J})^j}{j!}e^{2jM}\le e^{-4(J-M)} \exp(e^{-2(J-M)}) \le e^{-2(J-M)\1{a\neq b}} \varepsilon(J-M)\,.
%\end{equation}
%Using that for all real number $x$, we have $1+x\le \exp(x)$, \eqref{eq:uppbdcontiinterm} yields:
%\begin{equation}
%\calZ_{1, B}^{a, b}(J)
%\leq \exp(bB_1- 2(J-M)\1{a\neq b}+\varepsilon(J-M))\, ,
%\end{equation}
%so the upper bound holds in this case.
%If instead $H>M$, we make the upper bound
%\begin{align*}
%  \calZ_{1, B}^{a, b}(J)
%  & \leq \exp(bB_1) \times \sum_{j\ge 0} \frac{(e^{-2J})^j}{j!} e^{2jH}\\
%  & = \exp(bB_1+e^{-2J+2H})\, ,
% \end{align*}
%and we see that the upper bound holds also in this case.
\end{proof}

From Lemma \ref{th:lemcompZconti}, we derive the following lemma concerning the free energy densities.
 \begin{lem}\label{th:lemcompFconti}
 For every $M\ge 0$ and $J \ge 0$, we have:
 \begin{equation}
 \calF(J)\geq F_{\calN_1}(J+M)-\mathbb{E}[2(H-M)_+]. 
\end{equation}
and%, if $J-M\ge 0$,
\begin{equation}
 \calF(J) \leq F_{\calN_1}(J-M) + e^{2(M-J)} +\mathbb{E}\left[2(H-M)_+ +e^{2(H-J)}\right].
 \end{equation}
 \end{lem}
 
Replacing $J$ by $J-M$ (respectively $J+M$), Lemma \ref{th:lemcompFconti} immediately gives the following corollary.
 
\begin{cor}\label{th:corcompFconti}
 For every $M\ge 0$ and $J\ge 0$, we have:
 \begin{equation}
  F_{\calN_1}(J)\le \calF(J-M)+\mathbb{E}[2(H-M)_+], 
\end{equation}
%if $J-M\ge 0$, 
and
\begin{equation}
 F_{\calN_1}(J) \geq \calF(J+M)-e^{-2J}-\mathbb{E}\left[\left(2(H-M)+e^{2(H-M-J)}\right)\1{H>M}\right].
 \end{equation}
 \end{cor}

\begin{proof}[Proof of Lemma \ref{th:lemcompFconti}] Recall the observation that, under $\bbP$, sequence $h=(B_{n}-B_{n-1})_{n\in\bbN}$ is i.i.d. with law $\calN_1$. 
The lower bound is obtained using \eqref{eq:lowbdZconti} in \eqref{eq-rewriteAregroupconti}, taking the $\log$, dividing by $N$ and finally sending $N$ to infinity and using \eqref{eq:convFconti} and the law of large numbers.
The upper bound is obtained similarly, using \eqref{eq:uppbdZconti} instead of \eqref{eq:lowbdZconti}.
\end{proof}

%\begin{proof}[Proof of Corollary \ref{th:corcompFconti}]
%For the upper bound (respectively the lower bound), replace  $J$ by $J-M$ (respectively by $J+M$) in Lemma \ref{th:lemcompFconti}.
%\end{proof}

We are now ready to prove Theorem \ref{th:mainNv} when $\vartheta=1$: it is just a matter of making a choice for $M$, depending on $J$, and controlling the error terms appearing in Corollary \ref{th:corcompFconti}.

\begin{proof}[Proof of Theorem \ref{th:mainNv} when $\vartheta=1$]
We are going to use the following fact:% (see \cite{cf:CR81}): 
\begin{equation}\label{eq:contrH}
\bbP[H\ge \lambda] = O\left(\exp\left(-\frac{\lambda^2}{8}\right)\right)\, , \qquad \text{as } \lambda \to \infty \, .
\end{equation}
This can be shown by noticing that $H\ge \lambda$ implies that $\max_{[0, 1]} B_\cdot \ge \frac{\lambda}2$ or $\min_{[0, 1]} B_\cdot \le -\frac{\lambda}2$ and by using the fact that $\max_{[0, 1]} B_\cdot $ and $-\min_{[0, 1]} B_\cdot$ are distributed as the absolute values of standard Gaussian variables (by the reflection principle).
Equation \eqref{eq:contrH} yields in particular that $\bbE\left[e^{2H}\right]$ and $\bbE[H^2]$ are finite.

For every $J \ge 1$, let us set $M_J := 6 (\log J)^\frac12$. Then we have $\bbE\left[e^{2(H-M_J-J)}\right] = O(e^{-2J})$,
%\begin{equation}
%\bbE\left[e^{2(H-M_J-J)}\right] = O(e^{-2J}) \,  ,
%\end{equation}
and, using Cauchy-Schwarz inequality,
\begin{equation}
\bbE\left[(H-M_J)_+]\le \bbE[H \1{H>M_J}\right]\le \bbE[H^2]^\frac12  \bbP[H>M_J]^\frac12= O(J^{-2}) \, .
\end{equation}
Thus, Corollary \ref{th:corcompFconti} yields:
\begin{equation}
\calF(J+M_J)+O(J^{-2})
\le F_{\calN_1}(J) 
\le \calF(J-M_J)+O(J^{-2})\, .
\end{equation}
Recalling \eqref{eq:devFconti}, this concludes the proof of Theorem \ref{th:mainNv} in the case where $\vartheta=1$.
\end{proof}

\begin{rem}\label{rem:theta}
To conclude the present section, we present a way to adapt it to prove Theorem \ref{th:mainNv} for any $\vartheta\in(0, \infty)$.
It consists in replacing the Browian trajectory $B$ by $\vartheta B$. The variable $H$ is thus replaced by $\vartheta H$. In Lemma \ref{th:lemcompFconti} and in Corollary \ref{th:corcompFconti}, $\calF$ is replaced by $\calF_\vartheta$, the free energy density of the continuum RFIC with external field $\vartheta B$ (instead of $B$, see \eqref{eq:defZconti} and \eqref{eq:convFconti}). Using Brownian scaling and Poisson scaling, we see that for all $J\in\bbR$:
\begin{equation}
\calF_\vartheta ( J) =\vartheta^2 \calF(J+\log \vartheta).
\end{equation}
The proof of Theorem 1 can then be conducted \emph{mutatis mutandis}.
%Alternatively, we could have The first one consits in modifying \eqref{eq-rewriteAregroupconti} by considering the continuum partition function on $[0, N\vartheta^2]$ instead of $[0, N]$ and dividing this segment in blocks of length $\vartheta^2$, so $[(n-1)\vartheta^2, n\vartheta^2], n=1, \dots, N$. Lemmas \ref{th:lemcompZconti} and \ref{th:lemcompFconti} as well as Corollary \ref{th:corcompFconti} should then be rewritten for this new block length : 
%\begin{itemize}
%\item The random variable $H$ is replaced by $H_\vartheta:=\max_{0\le s\le t\le \vartheta^2} | B_s-B_t|$.
%\item In Lemma \ref{th:lemcompFconti} and Corollary \ref{th:corcompFconti}, $\calF$ gets multiplied by $\vartheta^2$.
%\item In Lemmas \ref{th:lemcompZconti} and \ref{th:lemcompFconti}, in the right-hand sides of the inequalities, $J$ is changed in $J- \log \vartheta$ : this comes from the fact that the Poisson parameter of a block of length $\vartheta^2$ is $\vartheta^2 e^{-2J}$. Consequently, in Corollary \ref{th:corcompFconti}, $J$ is replaced by $J+\log \vartheta$ in the right-hand sides of the inequalities.
%\end{itemize}
%In fact, since $H_\vartheta$ and $\vartheta H$ have the same law, both approaches yield the same generalizations of Lemma \ref{th:lemcompFconti} and Corollary \ref{th:corcompFconti}.

\end{rem}

\section{Gaussian approximation of the discrete RFIC partition function}

Our aim is to compare the free energy density of the RFIC with disorder distributions $\mu$ and $\calN_{\vartheta^2}$, in order to derive Theorems \ref{th:mainL} and \ref{th:mainLpoly} from Theorem \ref{th:mainNv}. We first state a lemma which bounds the difference between the free energy densities for different disorder laws, with respect to their Wasserstein 1-distance. 

\begin{lem}\label{lem-W1}
	Consider $\nu$ and $\nu'$, two laws on $\mathbb{R}$. Then, for all $J\ge 0$,
	\[|F_{\nu}(J) - F_{\nu'}(J)| \leq W_1(\nu  , \nu' ),\]
	where $W_1(\cdot, \cdot)$ is the Wasserstein 1-distance, i.e.:
	\[W_1(\nu , \nu' )= \inf_{(h, h')} \mathbb{E}[|h-h'|].\]
with the infimum running over all couplings $(h,h')$ of $\nu$ and $\nu'$.
\end{lem}

\begin{proof}
	Let us consider a coupling of $\nu$ and $\nu'$ and two sequences of random variables $h=(h_n)_{n\in\bbN}$ and $h'=(h_n')_{n\in\bbN}$ such that the sequence $((h_n, h_n'))_{n\in\bbN}$ is i.i.d.~and such that $(h_1, h_1')$ is distributed according to the considered coupling.
	Then, using $\sigma_n h_n \leq \sigma_n h_n'+|h_n-h_n'|$, we have:
	\[Z_{N,h}(J) \leq Z_{N,h'}(J) \exp\left(\sum_{n=1}^N |h_n-h_n'|\right).\]
	Taking the $\log$, dividing by $N$ and sending $N$ to infinity, \eqref{eq:convFdisc} yields that:
	\[F_{\nu}(J) \leq F_{\nu'}(J) + \mathbb{E}[|h_1-h_1'|],\]
	using the law of large numbers for the last term. The lemma follows via a minimization over the couplings of $\nu$ and $\nu'$, and by exchanging the roles of $\nu$ and $\nu'$.
\end{proof}

Since this lemma provides a bound which is uniform in $J$, it seems useless for our purpose. It will however become crucial once we have introduced the main idea of the present section, which is to make a \emph{coarse graining of the Ising chain}. This coarse-graining bears some similarities with the procedure used in the previous section to link the continuum RFIC to the Gaussian disorder discrete RFIC. It is also reminiscent of the block spin transform in the context of the renormalization group approach. We pick $L\in\bbN$. For every $N\in \bbN$, we ``rigidify'' the RFIC on $\{1, \dots, NL\}$ by restricting it to configurations which exhibit spin changes only at locations $nL, n=1, \dots, N$. Then the collection $(\sigma_{nL})_{1\le n\le L}$ is itself distributed as a RFIC on $\{1, .\dots, N\}$ with disorder sequence $h^L= (h_n^{L})_{n\in\bbN}$ defined by $h_n^{L}:=\sum_{i=(n-1)L+1}^{nL} h_i$. We finally observe that this sequence is i.i.d., we denote its law by $\mu^{\ast L}$, it is the law of the sum of $L$ independent variables with law $\mu$.
This coarse-graining approach will yield comparison estimates on the free energy densities $F_\mu$ and $F_{\mu^{\ast L}}$, see Lemma \ref{th:lemcompFdiscrete}. The lower bound is straight-forward to obtain, while the upper bound will require some work.  

With these estimates at hand, the proofs of Theorems \ref{th:mainL} and \ref{th:mainLpoly} will informally consist in the following approximations. We will let $L$ depend on $J$ and show, as $J\to\infty$ that:
\begin{equation}\label{eq:sketch23}
	F_{\calL}(J) 
	\approx \frac1L F_{\calL^{\ast L}}(J)\\
	\approx \frac1L F_{\calN_{L\vartheta^2}}(J)
	\approx F_{\calN_{\vartheta^2}}(J)
	\approx \frac{\vartheta^2}{2J} \, .
\end{equation}
The first and third approximations will be given by our coarse-graining estimates, while for the fourth one we will use Theorem \ref{th:mainNv}. The second one will be provided by Lemma \ref{lem-W1} together with the following lemma, which is a version of the central limit theorem with respect to the Wasserstein-1 distance. 
%Recall that we denote by $\calL^{\ast L}$ the law of the sum of $L$ independants variables with law $\calL$. 
%Recall that if $h_1, \dots, h_{L}$ are $L$ independent variables following law $\calL$, we denote by $\calL^{\ast L}$ the law of the sum $h_1+\dots+h_L$. 
%Recall also that the centered Gaussian laws satisfy the following relation: $\left(\calN_{\vartheta^2}\right)^{\ast L}=\calN_{L\vartheta^2}$.

\begin{lem}\label{lem-CLT} If $\calL$ has a finite third moment, then the sequence $\left(W_1\left(\mu^{\ast L}, \mathcal{N}_{L\vartheta^2}\right)\right)_{L\in\bbN^*}$ is bounded. 
If $\calL$ has a finite $p$-th moment, with $p\in [2, 3)$, then, as $L\to \infty$,
\begin{equation}
W_1\left(\mu^{\ast L}, \mathcal{N}_{L\vartheta^2}\right) = O\left(L^{\frac{3-p}2}\right)\, .
\end{equation}
%Assuming that $\mu$ has variance $\vartheta^2$, we have:
%\[W_1\left(\mu}^{\ast L}, \mathcal{N}_{L\vartheta^2}\right)=O(1),\quad \text{as } L \text{ goes to infinity}.\]

\end{lem}

\begin{proof}
This is equivalent to showing that the Wasserstein 1-distance between the law of $(h_1+\dots +h_n)/\sqrt{L}$ and $\calN_{\vartheta^2}$ is $O(L^{-\frac12})$, respectively $O(L^{1-\frac{p}2})$.
This is a known result, see \cite{cf:R09} for example.
% Theorem 5.3.1 \cite{cf:IL71} for a proof under the requirement that $\calL$ has a third moment.
\end{proof}

Lemmas \ref{lem-W1} and \ref{lem-CLT} yield that the error in the second approximation in \eqref{eq:sketch23} is $O(\frac{1}{L})$, respectively $O(\frac1{L^{\frac{p-1}{2}}})$. How large an $L$ we will take will depend on the error terms coming from the coarse-graining estimates, which we present now.

%[DELETE OR REWRITE THIS PARAGRAPH?] We expect that, as $J$ goes to infinity, the jumps are far apart, and therefore the potential can be thought of as being more and more like a Brownian motion, so somehow as if it had Gaussian increments $\calN_{\vartheta^2}$. We are going to show that the free energies of the discrete model with law $\calL$ and the discrete model with Gaussian environment $\calN_v$ are asymptotically equivalent. 

\subsection{Coarse-graining of the Ising chain}

As annouced, in the present section, we pick $L\in\bbN$ and, denoting by $\calL^{\ast L}$ the law of $h_1+\dots+h_L$ (recall that the $h_n$'s are i.i.d with law $\calL$), we will compare the free energy density $F_{\mu}$ with the free energy density  $F_{\calL^{\ast L}}$ by performing a ``coarse-graining'' of the Ising chain.

We need to slightly generalize the notation in \eqref{eq:defZN}. We introduce for positive integers $\ell$ and $\err$ satisfying $\ell\le \err$ the partition function on $\{\ell,\ell+1, \dots,  \err\}$:
\begin{equation}\label{eq-defZlr}
Z_{\ell, \err , h}^{a,b}(J):=
\sum_{\substack{\sigma\in \{-1, 1\}^{\{\ell-1, \ell, \dots, \err\}}\\ \text{s.t. } \sigma_{\ell-1}=a, \,  \sigma_\err=b} }
\exp\Big(-2J  \sum_{n=\ell}^{\err} \1{\sigma_n\neq \sigma_{n-1}}+\sum_{n=\ell}^{\err} \sigma_n h_n\Big)\, .
\end{equation}
Note that $Z_{N , h}^{a,b}(J)=Z_{1, N , h}^{a,b}(J)$ and that $Z_{\ell, \err , h}^{a,b}(J)= Z_{\err-\ell +1, \Theta^{\ell-1}h}^{a,b}(J)$ if we define $\Theta^{\ell-1}h$ as the sequence $(h_{\ell-1 + n}-h_{\ell-1})_{n\in\bbN}$.

The coarse-graining idea goes as follows: we define a sequence $h^L=(h^L_n)_{n\in\bbN}$ by setting $h^L_n:=\sum_{i=(n-1)L+1}^{nL} h_i$. As already stressed, the sequence $h^L$ is i.i.d. with law $\calL^{\ast L}$. For all $N\in\bbN$, we associate to any configuration $\sigma \in \{-1, 1\}^{\{0, 1, \dots, NL\}}$ the configuration $\sigma^L \in \{-1, 1\}^{\{0, 1, \dots, N\}}$ defined by:
\begin{equation}
\sigma^L_n = \sigma_{nL}, \qquad n=0, \dots, N.
\end{equation}
Multiple configurations $\sigma$ are sent to the same configuration $\sigma^L$, we rewrite the summation in \eqref{eq:defZN} according to this partitioning and we obtain:
\begin{equation}\label{eq:rewriteZNL}
Z_{NL, h}^{a, b}(J)=
\sum_{\substack{\sigma\in \{-1, 1\}^{\{0, 1, \dots, N\}}\\ \text{s.t. } \sigma_0=a,\,  \sigma_N=b} }
\prod_{n=1}^N Z_{(n-1)L+1, nL, h}^{\sigma_{n-1}, \sigma_n}(J) \, .
\end{equation}
Our aim is to compare this expression with
\begin{equation}
Z_{N, h^L}^{a, b}(J)
=\sum_{\substack{\sigma\in \{-1, 1\}^{\{0, 1, \dots, N\}}\\ \text{s.t. } \sigma_0=a,\,  \sigma_N=b} }
\prod_{n=1}^N \exp(-2J \1{\sigma_{n-1} \neq  \sigma_n}+\sigma_n h_n^L)\, ,
\end{equation}
so naturally we now want to compare the partition function on the \emph{block} $\{(n-1)L+1,\dots ,  nL\}$, i.e $Z_{(n-1)L+1, nL, h}^{\sigma_{n-1}, \sigma_n}(J)$, with $\exp(-2J \1{\sigma_{n-1} \neq  \sigma_n}+\sigma_n h_n^L)$. This will be provided by the following lemma, that without loss of generality we formulate only for the first block $\{1, \dots, L\}$. 
We introduce the random variable
\begin{equation}\label{eq:defHL}
H_L:=\max_{1\leq n \leq m \leq L} \Big| \sum_{i=n}^m h_i \Big| \, .
\end{equation}

\begin{lem}\label{th:lemcompZdiscrete}  
For every realization of $h$, for every $(a, b)\in\{-1, +1\}^2$, for every $L\in\N$, $J\ge 0$ and $M \ge 0$, we have
\begin{equation}\label{eq:lemcompZdiscretelowbd}
 Z_{L, h}^{a, b}(J) \ge \exp\left(-2J \1{a \neq  b}+b h_1^L\right)\, ,
\end{equation}
and%, if $J-M-\frac12 \log L\ge 0$,
\begin{equation}\label{eq:lemcompZdiscreteuppbd}
\begin{aligned}
Z_{L, h}^{a, b}(J) 
& \le \exp\Bigg[-2\left(J\!-\!M\!-\!\frac12\! \log\! L\right) \1{a \neq  b}+b h_1^L\\
& \qquad \qquad  +  L\left(4(H_L-M)_+ +  e^{2(M-J)}\right)\Bigg]\, .
\end{aligned}
\end{equation}
\end{lem}

\begin{proof}
For the lower bound, we simply lower bound the sum defining $Z_{L, h}^{a, b}(J)$ in \eqref{eq:defZconti} by the term corresponding to $\sigma=(a, b, b, \dots, b)$.

We now turn to the upper bound. For given $\sigma\in\{-1, +1\}^{\{0, 1, \dots, L\}}$ such that $\sigma_0=a, \sigma_L=b$, denoting $j=\sum_{n=1}^L \1{\sigma_n\neq \sigma_{n-1}}$ the number of spin flips in $\sigma$, we observe that:
\begin{equation}
\sum_{n=1}^L \sigma_n h_n \leq b h^L_1 + 2 j H_L
\end{equation}
Furthermore, there are at most $\binom{L}{j}$ configurations in $\{-1, +1\}^{\{0, 1, \dots, L\}}$ having $j$ spin flips, with the convention $\binom{L}{j}=0$ if $j>L$. 
Therefore:
\begin{equation}\label{eq:proofuppboundZdisc}
Z_{L, h}^{a, b}(J) 
\le \sum_{j} \binom{L}{j} \exp(-2jJ+b h_1^L + 2j H_L) \, , 
\end{equation}
where the sum is over all non-negative even integers if $a=b$, over all positive odd integers if $a\neq b$. In the first case, let us upper bound the sum by the sum over all non-negative integers and in the second case let us bound it by the sum over all positive integers and use $\binom{L}{j}\le L \binom{L}{j-1}$. Using the binomial formula, we get:%Since $J-M-\frac12 \log L\ge 0$, 
\begin{equation}\label{eq:proofuppboundZdiscrewrite}
Z_{L, h}^{a, b}(J) \le \exp\left(bh_1^L -2(J-H_L-\frac12 \log L) \1{a\neq b}\right) \times  (1+e^{2(H_L-J)})^L
\end{equation} 
We use $H_L\le M+(H_L-M)_+$ to bound $H_ L\1{a\neq b}$ and we use it again to get:
\begin{equation}
1+e^{2(H_L-J)} \le e^{2(H_L-M)_+} (1+e^{2(M-J)})
\end{equation}
Using the fact that $1+x\le e^x$ for any real number $x$, we get:
\begin{equation}
Z_{L, h}^{a, b}(J) 
\le \exp\left[-2\left(\!J\!-\!M\!-\!\frac12\! \log\! L\!\right) \1{a \neq  b}+b h_1^L
 +  2(L+1)(H_L-M)_+ + L e^{2(M-J)}\right]\, .
\end{equation}
We bound $L+1$ by $2L$ to conclude the proof of the upper bound.
\end{proof}

From Lemma \ref{th:lemcompZdiscrete}, we derive the following lemma concerning the free energy densities. 
%Recall the assumptions [on $\calL$ in the introduction].
 \begin{lem}\label{th:lemcompFdiscrete}
  %There exist positive constants $c$ and $C$ depending on $\calL$ such that 
  For every $L\in\N$, $J\ge 0$ and $M\ge 0$, we have:
\begin{equation}\label{eq:lemcompFdiscretelowbd}
 F_{\calL}(J) \geq \frac1LF_{\calL^{\ast L}}(J). 
\end{equation}
and%, if $J-M-\frac12 \log L\ge 0$,
\begin{equation}\label{eq:lemcompFdiscreteuppbd}
 F_{\calL}(J) 
 \leq \frac1L F_{\calL^{\ast L}}(J\!-\!M\!-\!\frac12\!\log L) + 4 \bbE_\mu\left[(H_L-M)_+ \right]+  e^{2(M-J)}\,.
 \end{equation}
 \end{lem}
 
\begin{proof}%[Proof of Lemma \ref{th:lemcompFdiscrete}]
The lower bound is immediate, using \eqref{eq:lemcompFdiscretelowbd} in \eqref{eq:rewriteZNL}, taking the $\log$, dividing by $NL$ and finally sending $N$ to infinity with the help of \eqref{eq:convFdisc}.
For the upper bound, we do the same, using \eqref{eq:lemcompFdiscreteuppbd} instead of \eqref{eq:lemcompFdiscretelowbd} and using in addition the law of large numbers. 
%We get:
%\begin{align*}
% F_{\calL}(J) 
% & \leq \frac1L F_{\calL^{\ast L}}(J\!-\!M\!-\!\frac12\!\log L) + \frac1L \varepsilon(J\!-\!M\!-\!\frac12\!\log L)\\
% & \qquad +\frac1L \bbE\left[\left(L\log 2+2\sum_{n=1}^{L} |h_n|\right)\1{H_L>M}\right]\, .
%\end{align*}
%To handle the last term, we use the linearity of the expectation.
% and we observe that for all $n=1, \dots, L$, 
%$\mathbb{E}[ |h_n|\1{H_{L}>M}]=\mathbb{E}[ |h_1|\1{H_{L}>M}]$ .
\end{proof}

Replacing $J$ by $J\!+\!M\!+\!\frac12\!\log L$, Lemma \ref{th:lemcompFdiscrete} yields the following corollary.
 
 \begin{cor}\label{th:corcompFdiscrete}
 For every $L\in\N$, $J\ge 0$ and $M \ge 0$,
 \begin{equation}
 \frac1L F_{\calL^{\ast L}}(J)
 \geq F_{\calL}(J\!+\!M\!+\!\frac12\!\log L)
 - 4 \bbE_\mu\left[(H_L-M)_+\right] -  \frac{e^{-2J}}{L}\, .
 \end{equation}
 %b_{L,M}(J\!+\!M\!+\!\frac12\!\log L)
 \end{cor}

\begin{rem}
In the proofs of Theorem \ref{th:mainL} of \ref{th:mainLpoly}, we will take $L$ depending on $J$ and large (as $J$ goes to infinity). We will also let $M$ depend on $J$, typically larger than scale $\sqrt{L}$, so as to make the event $\{H_L> M\}$ unlikely,
%or "in order to" ?
but also smaller than scale $J$ so that the bounds remain interesting for our purpose.%, see \eqref{eq:choiceLMt}.
\end{rem}

%To handle the last term, we use the linearity of the expectation, and then Cauchy-Schwarz inequality. For all $n=1, \dots, L$:
%\begin{align*}
%\mathbb{E}\Big[ |h_n|\1{H_{L}>M}\Big]
%\leq  \mathbb{E}\left[h_n^2\right]^\frac12 \bbP\Big[H_{L}>M\Big]^\frac12 = \vartheta \; \bbP\Big[H_{L}>M\Big]^\frac12 \, .
%\end{align*}
%Hence:
%\[\frac1L\bbE\left[(L\ln 2+2\sum_{n=1}^{L} |h_n|)\1{H_{L}>M}\right]\leq (\ln 2) \; \bbP\left[H_{L}>M\right] + 2\vartheta \; \bbP\left[H_{L}>M\right]^\frac12.\]
%We conclude the proof by noticing that $\bbP\left[H_{L}>M\right]\in[0, 1]$.

%\subsection{Comparison of the general case with the Gaussian case}

%[DELETE SECTION ] Recall that our aim is to establish a comparison between the free energy density of the Ising chain with disorder law $\calL$ and the free energy density corresponding to $\calN_{\vartheta^2}$, in order to derive Theorem \ref{th:mainL} from Theorem \ref{th:mainNv}. Using our coarse-graining result, we will be able to reduce this to showing the same comparison but with $\calL$ replaced by $\calL^{\ast L}$ and with $\calN_{\vartheta^2}$ replaced by $\left(\calN_{\vartheta^2}\right)^{\ast L}=\calN_{L\vartheta^2}$, with $L$ large. This comparison will be performed using CLT estimates and a lemma which we state now, that allows us to compare the free energy densities for distinct choices of the disorder law $\calL$.

\subsection{Proof of Theorem \ref{th:mainL}}

\begin{lem}\label{th:lemboundEmuexpo}
Assume that $\mu$ admits some finite exponential moments and recall that it is centered with variance denoted $\vartheta^2$. Then, there exists a positive constant $c$ such that, for all $L\in\bbN$ and all $M>0$ satisfying $M\le c L$, we have
\begin{equation}\label{eq:lemboundE+expo}
\bbE_\mu[(H_L-M)_+]  \le \frac{4 \vartheta^2 L^3}{M} \exp\left(-\frac{M^2}{4\vartheta^2 L} \right) \, .
\end{equation}
%\begin{equation}\label{eq:lemboundEmuexpo}
%\bbP_\mu[H_L>M]  \le 2 L^2 \exp\left(-\frac{M^2}{4\vartheta^2 L} \right) \, .
%\end{equation}
%and
%\begin{equation}\label{eq:lemboundEmuexpoconseq}
%\bbE_\mu[(H_L-M)_+]  \le \vartheta \sqrt{2} L^2 \exp\left(-\frac{M^2}{8\vartheta^2 L} \right) \, .
%\end{equation}
%\le (\log 2 + 2 \vartheta) \sqrt{2} L \exp\left(\frac{C t^2 L - tM}{2} \right)
\end{lem}

\begin{proof}
Let us denote by $\varphi: \bbR\to \bbR, t\mapsto \bbE_\mu[\exp(th_1)]$ the moment generating function of $\mu$. Developing it around 0, we see that there exists a positive constant $c'$ such that:
\begin{equation}\label{eq:devexpomom}
\forall t\in[-c', c'] \, , \quad \varphi(t)\leq \exp(\vartheta^2 t^2)\, .
\end{equation}
We pick any $t\in(0, c')$. Recalling the definition of $H_L$ in \eqref{eq:defHL}, we observe that:
\begin{equation}
\exp(t H_L) \le \sum_{1\le n \le m\le L}  \exp\left(\sum_{i=n}^m h_i \right) + \exp\left(-\sum_{i=n}^m h_i \right) 
\end{equation}
which, by taking the expectation under $\bbP_\mu$ and using \eqref{eq:devexpomom}, yields:
\begin{equation}
\bbE_\mu[\exp(t H_L)] \le \sum_{1\le n \le m\le L} \varphi(t)^{(m-n+1)}+\varphi(-t)^{(m-n+1)} \le  2 L^2\exp(\vartheta^2 t^2 L)\, .
\end{equation}
Using the fact that $(H_L-M)_+\le \frac1t  \exp(t(H_L-M))$, we obtain:
\begin{equation}\label{eq:lemboundE+expowitht}
\bbE_\mu[(H_L-M)_+ ] \le \frac{2L^2}t  \exp(\vartheta^2 t^2 L-Mt) \, .
\end{equation}
Now, we see that the lemma holds with $c:=2\vartheta^2 c'$. Indeed, when $M\le c L$, take $t=\frac{M}{2 \vartheta^2 L}$, which is not larger than $c'$, to derive \eqref{eq:lemboundE+expo} from \eqref{eq:lemboundE+expowitht}.
\end{proof}

We are now ready to prove Theorem \ref{th:mainL}.

\begin{proof}[Proof of Theorem \ref{th:mainL}]

In this proof, we use $\tM$ as a short-hand for $M+\frac12 \log L$.
Using the lower bound in Lemma \ref{th:lemcompFdiscrete} for law $\mu$, then Lemma \ref{lem-W1} and finally Corollary \ref{th:corcompFdiscrete} for law $\calN_{\vartheta^2}$, we derive a lower bound on $F_{\calL}(J)$:
\begin{equation}\label{eq:proofth2lowbound}
\begin{aligned}
 F_{\calL}(J) 
 & \geq \frac1L F_{\calL^{\ast L}}(J) \\
 & \geq \frac1L F_{\calN_{L\vartheta^2}}(J) - \frac1L W_1(\calL^{\ast L},\calN_{L\vartheta^2}) \\
 & \geq F_{\calN_{\vartheta^2}}(J+\tM)- 4 \bbE_{\calN_{\vartheta^2}}\left[(H_L-M)_+\right] -  \frac{e^{-2J}}{L}- \frac1L W_1(\calL^{\ast L},\calN_{L\vartheta^2}) \, .
\end{aligned}
\end{equation}
Recall Theorem \ref{th:mainNv}. Using the fact that $\frac1{1+x}\ge 1-x $ for all $x\in\bbR^+$, we have 
\begin{equation}\label{eq:proofth2lowboundhomog}
F_{\calN_{\vartheta^2}}(J+\tM) \ge \frac{\vartheta^2}{2J} - \frac{\vartheta^2 \tM}{2J^2} +O\left(\frac{(\log J)^\frac12}{J^2}\right)\, .
\end{equation}

%Since $\calN_v$ satisfies the assumptions made on $\calL$, this holds also if $\calL$ is replaced by $\calN_v$.

Similarly, we derive an upper bound on $F_{\calL}(J)$. 
Applying the upper bound in Lemma \ref{th:lemcompFdiscrete} for law $\mu$, then Lemma \ref{lem-W1} and finally the lower bound in Lemma \ref{th:lemcompFdiscrete} for law $\calN_{\vartheta^2}$, we have:
\begin{equation}\label{eq:proofth2upperbound}
\begin{aligned}
 F_{\calL}(J)
 & \leq \frac1L F_{\calL^{\ast L} }(J-\tM) + 4 \bbE_\mu\left[(H_L-M)_+ \right]+  e^{2(M-J)}\\
 & \leq \frac1L F_{\calN_{L\vartheta^2}}(J- \tM) + \frac1L W_1(\calL^{\ast L},\calN_{L\vartheta^2})
  + 4 \bbE_\mu\left[(H_L-M)_+ \right]+  e^{2(M-J)} \\
 & \leq F_{\calN_{\vartheta^2}}(J-\tM) +\frac1L W_1(\calL^{\ast L},\calN_{L\vartheta^2})
  + 4 \bbE_\mu\left[(H_L-M)_+ \right]+  e^{2(M-J)}\, .
\end{aligned}
\end{equation}
Recall Theorem \ref{th:mainNv}. We assume that $\tM\le \frac{J}2$. Using the fact that $\frac1{1-x}\le 1+2x $ for all $x\in[0, \frac12]$, we have 
\begin{equation}
F_{\calN_{\vartheta^2}}(J-\tM) \le \frac{\vartheta^2}{2J} + \frac{\vartheta^2 \tM}{J^2} + O\left(\frac{(\log J)^\frac12}{J^2}\right)\, .
\end{equation}
Now, we make a choice for $L$ and $M$. 
We set, for $J>1$:
\begin{equation}\label{eq:choiceLMtexpo}
L_J=\left\lfloor \frac{J^{\frac43}}{(\log J)^\frac13}\right\rfloor , \qquad M_J= 6 \vartheta J^{\frac23} (\log J)^\frac13,\qquad  
\end{equation}
With this choice, we have $e^{2(M-J)}=O(J^{-2})$ and, according to Lemma \ref{th:lemboundEmuexpo}, both $\bbE_{\calN_{\vartheta^2}}\left[(H_{L_J}-M_J)_+\right]$ and $\bbE_{\mu}\left[(H_L-M)_+\right]$ are $O(J^{-2})$. According to Lemma \ref{lem-CLT}, $W_1(\calL^{\ast L_J},\calN_{L_J\vartheta^2})=O(1)$. Theorem \ref{th:mainL} follows.
\end{proof}

\subsection{Proof of Theorem \ref{th:mainLpoly}}% Polynomial integrability assumption}

Notice that in the proof of the lower bound in Theorem \ref{th:mainL}, we only used the bound on $\bbE_{\calN_{\vartheta^2}}\left[(H_L-M)_+\right]$ but not on $\bbE_{\mu}\left[(H_L-M)_+\right]$. Thus, it is possible to carry this proof out in the same way in the case of polynomial integrability. The only adaptation that is needed is to take into account the disjunction in Lemma \ref{lem-CLT}.

\begin{proof}[Proof of the lower bound in Theorem \ref{th:mainLpoly}]
We use the lower bounds \eqref{eq:proofth2lowbound} and \eqref{eq:proofth2lowboundhomog} establised in the proof of Theorem \ref{th:mainL}. For the case $\mu\in L^3$, we choose $L_J$ and $M_J$ as in the proof of Theorem \ref{th:mainL}, see \eqref{eq:choiceLMtexpo}. For the case $\mu\in L^p$, with $p\in[2, 3)$, we set, for $J>1$:
\begin{equation}\label{eq:choiceLMtpolylow}
L_J=\left\lfloor \frac{J^{\frac{4}{p}}}{(\log J)^\frac1{p}}\right\rfloor , \qquad M_J= 6 \vartheta J^{\frac2{p}} (\log J)^{\frac{p-1}{2p}}\, .
\end{equation}
With these choices, we have $e^{2(M_J-J)}=O(J^{-2})$ and, according to Lemma \ref{th:lemboundEmuexpo}, $\bbE_{\calN_{\vartheta^2}}\left[(H_{L_J}-M_J)_+\right]=O(J^{-2})$. We also use Lemma \ref{lem-CLT} to bound $W_1(\calL^{\ast L_J},\calN_{L_J\vartheta^2})$. The proof of the lower bound in Theorem \ref{th:mainLpoly} follows.
\end{proof}

Now, we turn to the proof of the upper bound in Theorem \ref{th:mainLpoly}. 
\begin{lem}\label{th:lemboundEmupoly}
Assume that $\mu\in L^p$, with $p\in [2, \infty)$. Then, for every $L\in \bbN$ and every $M>0$, we have:
\begin{equation}\label{eq:lemboundPmuHpoly}
\bbP_\mu[H_L>  M] = O \left( \frac{L^{\frac{p}2}}{M^{p}}\right)\, , 
\end{equation} 
and
\begin{equation}\label{eq:lemboundEmuH+poly}
\bbE_\mu[(H_L -  M)_+] = O \left( \frac{L^{\frac{p}2}}{M^{p-1}}\right)\,.
\end{equation} 
\end{lem}
\begin{proof}
Let us introduce $S_n:=\sum_{i=1}^n h_i$ and observe that $(S_n)_{n\in\bbN}$ is a martingale and $H_L \le 2  \sup_{1\le n\le L} |S_n|$. Doob's inequality in $L^p$ yields:
\begin{equation}\label{eq:Doobs}
\bbE[H_L^p] \le 2^p \bbE[(\sup_{1 \le n\le L} |S_n|)^p]\le 2^p \left(\frac{p}{p-1}\right)^p \bbE[|S_L|^p].
\end{equation}
Now, Rosenthal's inequality for independent centered random variables yields: 
$\bbE[|S_L|^{p}] = O(L^{\frac{p}2})$, hence we have shown that $\bbE[H_L^p] = O(L^{\frac{p}2})$.
The first statement of the lemma follows by Markov inequality. To show the second one, use instead the following inequality:
\begin{equation}\label{eq:MarkovLp}
\bbE[(H_L-M)_+]\le \bbE[H_L \1{H_L>M}]\le \bbE\left[H_L \left(\frac{H_L}{M}\right)^{p-1}\1{H_L>M} \right]\le \frac{\bbE[H_L^p]}{M^{p-1}}\, .
\end{equation}
\end{proof}

With this lemma at hand, it is possible to adapt the proof of Theorem \ref{th:mainL} to obtain an upper bound on $F_\mu(J)$. However, we present a different approach that slightly improves the final result. We modify the upper bounds in Lemmas \ref{th:lemcompZdiscrete} and \ref{th:lemcompFdiscrete} in the following way.

\begin{lem}\label{th:lemcompZandFdiscreterevisited}  
For every realization of $h$, for every $(a, b)\in\{-1, +1\}^2$, for every $L\in\N$, $J\ge 0$ and $M \ge 0$, we have%, if $J-M-\frac12 \log L\ge 0$,
\begin{equation}\label{eq:lemcompZdiscreteuppbdrevisited}
\begin{aligned}
Z_{L, h}^{a, b}(J) 
& \le \exp\Bigg[-2\left(J\!-\!M\!-\!\frac12\! \log\! L\right) \1{a \neq  b}+b h_1^L\\
& \qquad \qquad + \left(L\log 2+ 2 \sum_{n=1}^L |h_n|\right) \1{H_{L}>M} + Le^{2(M-J)}
\Bigg]\, ,
\end{aligned}
\end{equation}
and, consequently,
\begin{equation}\label{eq:lemcompFdiscreteuppbdrevisited}
 F_{\calL}(J) 
 \leq \frac1L F_{\calL^{\ast L}}(J\!-\!M\!-\!\frac12\!\log L) + \bbE_\mu\Big[(\log 2+\frac2L \sum_{n=1}^L |h_n|) \1{H_{L}>M}\Big] +e^{2(M-J)}\,.
 \end{equation}
 %where $E_\mu(L, M):=\bbE_\mu\Big[(\log 2+\frac2L \sum_{n=1}^L |h_n|) \1{H_{L}>M}\Big]$.
 \end{lem}
 
\begin{proof}
If $H\le M$, we use \eqref{eq:proofuppboundZdiscrewrite} to see that \eqref{eq:lemcompZdiscreteuppbdrevisited} holds. 
If instead $H_L>M$, we use the bound 
\begin{equation}
\sum_{n=1}^L \sigma_n h_n \leq b h^L_1 + 2 \sum_{n=1}^L |h_n|
\end{equation}
and the fact that if $a\neq b$ then $\sigma$ exhibits at least one spin flip. We get:
\begin{equation}
Z_{L, h}^{a, b}(J) 
\le 2^L \exp\left(-2J \1{a\neq b}+ b h_1^L + 2 \sum_{n=1}^L |h_n|\right) 
\end{equation}
so \eqref{eq:lemcompZdiscreteuppbdrevisited} holds also in this case.

To derive \eqref{eq:lemcompFdiscreteuppbdrevisited}, reason as for the upper bound in Lemma \ref{th:lemcompFdiscrete}.
\end{proof}
 
\begin{lem}\label{th:lemboundEmupolyrevisited}
Assume that $\mu\in L^p$, with $p\in [2, \infty)$. Then:
\begin{equation}
\bbE_\mu\Big[\left(\log 2+\frac2L \sum_{n=1}^L |h_n|\right) \1{H_{L}>M}\Big] = O \left( \frac{L^{\frac{p}2}}{M^{p}} + \frac1{M^{p-1}} \right)\, .
\end{equation} 
\end{lem}
\begin{proof}
% O \left( \frac{L^{\frac{p}2}}{M^{p}} + \frac1{M^{p-1}} \right)\, .
Recall the bound $\bbP[H_L>M]=O\left(\frac{L^{\frac{p}2}}{M^p}\right)$ which controls the first term. 
We now bound for each $n=1, \dots, L$ the term $\bbE[|h_n|\1{H_{L}>M}]$. We set 
\begin{equation}
H_{L}^{(n)}:= \sup_{1\le m\le m' \le L} \left|\sum_{i\in\{m, \dots, m'\}\setminus\{n\}} h_i \right|
\end{equation}
and we observe that if $|h_n|\le \frac{M}2$ and $H_L>M$ then $H_{L}^{(n)}>\frac{M}2$, so:
\begin{equation}\label{eq:boundh1HLM}
\bbE[|h_n|\1{H_{L}>M}]\le  \bbE[|h_n|\1{H_{L}^{(n)} > \frac{M}2}] + \bbE[|h_n| \1{|h_n|\ge \frac{M}2}]\, .
\end{equation} 
We handle the first term in the right-hand side of \eqref{eq:boundh1HLM} by noticing that $h_n$ and $H_{L}^{(n)}$ are independent, with $H_{L}^{(n)}$ distributed as $H_{L-1}$, so that \eqref{eq:lemboundPmuHpoly} can be used. For the second one we write, proceeding as in \eqref{eq:MarkovLp}:
\begin{equation}
\bbE[|h_n| \1{|h_n|\ge \frac{M}2}]
\le  \frac{\bbE[|h_n|^p]}{\left(\frac{M}2\right)^{p-1}}\, .
\end{equation} 
The lemma follows.
\end{proof}

\begin{proof}[Proof of the upper bound in Theorem \ref{th:mainLpoly}]
We proceed as in the proof of the upper bound in Theorem \ref{th:mainL}, but using Lemma \ref{th:lemcompZandFdiscreterevisited} instead of Lemma \ref{th:lemcompFdiscrete}: the error term $4\bbE_\mu[(H_L-M)_+]$ in \eqref{eq:proofth2upperbound} is replaced by $\bbE_\mu\Big[(\log 2+\frac2L \sum_{n=1}^L |h_n|) \1{H_{L}>M}\Big]$. We assume that $\mu$ is in $L^p$, with $p\in[2, \infty)$ and we make a choice for $L$ and $M$, depending on $p$.
For $p\in[3, \infty)$, we set $\eta_p=\frac{4p}{3p+2}$ and for $J\ge 1$:
\begin{equation}\label{eq:choiceLMpoly>3}
L_J=\left\lfloor J^{\eta_p}\right\rfloor , \qquad M_J= J^{2-\eta_p}\, .
\end{equation}
For $p\in[2, 3)$, we set $\eta_p= \frac{2p(p-1)}{p^2+p-1}$ and for $J\ge 1$:
%If $\frac{3+\sqrt{5}}{2}< p< 3$, we take:
\begin{equation}\label{eq:choiceLMpoly<3}
L_J=\left\lfloor J^{\frac{2\eta_p}{p-1}}\right\rfloor , \qquad M_J= J^{2-\eta_p}\, .
\end{equation}
With these choices, we have $e^{2(M_J-J)}=O(J^{-2})$, and according to Lemma \ref{th:lemboundEmupolyrevisited}, the error term $\bbE_\mu\Big[(\log 2+\frac2{L_J} \sum_{n=1}^{L_J} |h_n|) \1{H_{L_J}>M_J}\Big]$ is $O(J^{-\eta_p})$, using the identity $(2-\eta_p)p- \eta_p \frac{p}2=\eta_p$ for $p\in[3, \infty)$, the identity $(2-\eta_p)p- \frac{2\eta_p}{p-1} \frac{p}2=\eta_p$ for $p\in[2, 3)$ and the inequality $(p-1)(2-\eta_p)\ge \eta_p$ in both cases. 
We finally use Lemma \ref{lem-CLT} to bound $W_1(\calL^{\ast L_J},\calN_{L_J\vartheta^2})$ and we obtain the upper bound in Theorem \ref{th:mainLpoly}.
\end{proof}

\section{Acknowledgements}
The author thanks Raphael L. Greenblatt and Giambattista Giacomin for careful rereading of the manuscript.

\end{document}